\documentclass[12pt,a4paper]{amsart}
\usepackage[latin1]{inputenc}
\usepackage[english]{babel}
\usepackage{amsmath,amssymb,amsthm}

\usepackage{cite}

\theoremstyle{plain}
\newtheorem{theorem}{Theorem}[section]

\newtheorem{corollary}[theorem]{Corollary}

\theoremstyle{definition}

\newtheorem{remark}[theorem]{Remark}

\begin{document}

\title[Klurman] {V. Markov's problem for monotone polynomials}

\keywords{Markov inequality; Bernstein inequality; Jacobi polynomial.}

\subjclass[2000]{ 41A17. }

\author[Klurman]{Oleksiy Klurman}

\address{Department of Mathematics,
University of Manitoba, R3T2N2~Winnipeg,
Canada} \email{\texttt{lklurman@gmail.com}}

\date{March 6, 2012.}

\begin{abstract}
We consider the classical  problem of estimating norm of the derivative of algebraic polynomial via the norm of polynomial itself. The corresponding extremal problem for general polynomials in uniform norm was solved by V. Markov. In this note we solve analogous problem for monotone polynomials. As a consequence, we find exact constant in Bernstein inequality for monotone polynomials. \end{abstract}
\maketitle

\begin{section}{Introduction}

We consider the following extremal problem:

\begin{quote}
 {\it For a given norm $\|\cdot\|$, determine the best constant $A_n$
such that the inequality $$\|P_n'\|\le A_n\|P_n\|$$ holds for all
$P_n\in\mathbb{P}_n$, i.e.,
$$A_n=\sup_{P_n\in\mathbb{P}_n}\frac{\|P'_n\|}{\|P_n\|}.$$}
\end{quote}

The first result in this area appeared in $1889$. It is the
well- known A. Markov's inequality, namely:
\begin{theorem}\label{amarkov}(A. Markov). For every polynomial $P_n\in \mathbb{P}_n,$ the
following inequality holds:
\begin{equation}
\label{eq:amarkov}
\|P'_n\|\le n^2\|P_n\|.
\end{equation}
The equality holds if and only if $P_n=cT_n,$ where $T_n$ is the Chebyshev polynomial of the first kind, that is $T_n (x)=\cos(n\arccos x)$ for $x\in[-1,1].$
 \end{theorem}
 
 By $\triangle_n$ we denote the set of all monotone polynomials of
degree $n$ on $[-1,1].$ 
In $1926,$ S. Bernstein~\cite{MR1512353} pointed out that Markov's inequality for
monotone polynomials is not essentially better than for all
polynomials, in the sense, that the order of $\sup_{P_n\in\triangle_
n}\|P'_n\|/\|P_n\|$ is $n^2$. He proved his result only for
odd $n$. In $2001,$ Qazi \cite{MR1835375} extended Bernstein's idea to include polynomials of even degree. Next theorem contains their results:

 \begin{theorem}\label{bernquaz}(Bernstein \cite{MR1512353}, Qazi \cite{MR1835375}).
\[
\sup_{P_n\in\triangle_ n}\frac{\|P'_n\|}{\|P_n\|}=\left\{
\begin{array}{ll}
 \frac{(n+1)^2}{4} , & \mbox{\rm if } n=2k+1 ,\\
 \frac{n(n+2)}{4} , & \mbox{\rm if }  n=2k.

\end{array}
\right.
\]
\end{theorem}

V. Markov  investigated a more general problem:
\begin{quote} if $k_0,
k_1,..., k_n$ are given constants and $P_n(x)=\sum\limits_{i=0}^n
a_i x^i$ satisfies $\|P_n\|=1,$  what is the precise bound for the
linear form $\sum\limits_{i=0}^n a_ik_i$ ?\end{quote} By suitably choosing the constants $k_i$ the linear
form can be made equal to any derivative of $P_n(x)$ at any
preassigned point.  

{\bf V. Markov's problem.}
\begin{quote}
 {\it Let $x_0\in [-1,1]$ be a fixed point. For $0\le k\le n,$
 find the maximum value of $|P^{(k)}(x_0)|$ over all $P_n\in\mathbb{P}_n$ such that $\|P_n\|=1.$}
\end{quote}
 The problem was studied more completely and in considerably shorter
 way by Gusev ~\cite{MR0197647} with the help of a method developed by
 Voronovskaja, who solved this problem for the case $k=1$, (see ~\cite{MR0267057}).

In this note, we give a solution of an analogous problem for the case of
monotone polynomials and  $k=1$, namely the following
problem is considered:

{\bf Problem.}
\begin{quote}
{\it Let $x_0\in [-1,1]$ be a fixed point.
 Find the maximum value of $|P'(x_0)|$ over all monotone polynomials $P_n\in\mathbb{P}_n$ such that $\|P_n\|=1.$}
\end{quote}

As a consequence, we obtain a simple proof of the main result from  \cite{MR1835375} as well as sharp Bernstein's inequality for monotone polynomials.
\end{section}
\begin{section}{Proof of Main Result}

 In order to formulate the main result the following three types of polynomials are needed :
   \begin{align*}
 &S_k (x):=(1+x)\sum\limits_{l=0}^{k} (J^{(0,1)}_l (x))^2;\\
&H_k (x):=(1-x^2)\sum\limits_{l=0}^{k-1} (J_l ^{(1,1)} (x))^2;\\
&F_k(x):=\sum\limits_{l=0}^{k} (J_l ^{(0,0)} (x))^2.\\
\end{align*}

\begin{theorem}\label{oklurman}
Let $x_0$ be fixed point in the interval $[-1,1].$ Then, for every
$P_n\in\triangle_n,$ $n\ge 1,$ the following sharp inequality holds:
$$|P'_n(x_0)|\le 2\max(S_k (x_0),S_k (-x_0))\|P_n\|,$$ for $n=2k+2,$ $k\ge 0$, and
$$|P'_n (x_0)|\le 2\max(F_k(x_0),H_k(x_0))\|P_n\|,$$ for $n=2k+1,$ $k\ge 0.$
\end{theorem}

\begin{proof}
 We start with the solution of the following problem.
 Fix $x_0\in [-1,1],$ find the maximum value of $$S(P,x_0):=\frac
{P(x_0)}{\int\limits_{-1}^{1}P(x)dx},$$ over $P\in\mathbb{P}^+_n$,
where $\mathbb{P}^+_n$ denotes the set of all nonnegative on $[-1,1]$ polynomials of degree at most $n.$
 In what follows, we assume that $x_0\in (-1,1).$ All the results can be extended to $x_0=1$ and $x_0=-1$ by continuity.

 Note, that this maximum value is attained because of the
sequentially compactness of our set $\mathbb{P}^+_n$.

 Let us denote
by $P^*(x)$ an extremal polynomial from ${P}_{n,1}^+$ with the largest degree and the maximal number of zeros inside the interval $[-1,1].$ In other words, if $$Q(x)=argmax_{P\in{P}_{n,1}^+}S(P,x_0),$$ then $\deg P^*\ge\deg Q,$ and the number of zeros of $Q$ inside $[-1,1]$  $\le$ than the number of zeros of $P^*.$

We first prove $\deg{(P^*)} =n.$ Indeed, if $\deg{(P^*)}\le n-1$ consider two polynomials:

$$P_1 (x)=(1-x)P^*(x),$$ $$P_2 (x)=(1+x)P^*(x).$$ None of them can be extremal, hence,
$$\frac
{P^*(x_0)}{\int\limits_{-1}^{1}P^*(x)dx}>\frac{(1-x_0)P^*(x_0)}{\int\limits_{-1}^{1}(1-x)P^*(x)dx},$$
and $$\frac
{P^*(x_0)}{\int\limits_{-1}^{1}P^*(x)dx}>\frac{(1+x_0)P^*(x_0)}{\int\limits_{-1}^{1}(1+x)P^*(x)dx}.$$
Multiplying both inequalities by common denominators and adding the results up we get $$\int\limits_{-1}^{1}P^*(x)dx>\int\limits_{-1}^{1}P^*(x)dx,$$ that provides a contradiction, and so $\deg (P^*)=n.$ The next step is to show, that all zeros
of $P^*(x)$ lie in the interval $[-1,1].$ Suppose that this
is not the case and
 write $P^*(x)=P_1(x)P_2(x),$
where all zeros of $P_1$ lie in $[-1,1]$ and $P_2(x)>\delta
>0,$ for all $x\in [-1,1],$ and $\deg{(P_2)}\ge 1$.
Note, that for every fixed polynomial $h$, $\deg{(h)}\le \deg{(P_2)}$
and sufficiently small $t$ all polynomials of the form
$Q(x)=P^*(x)+th(x)P_1(x)$ belong to $\mathbb{P}^+_n$. Hence, $t=0$
should be a point of local minimum of the function
$$g(t)=\frac
{P^*(x_0)+th(x_0)P_1(x_0)}{\int\limits_{-1}^{1}(P^*(x)+th(x)P_1
(x))dx}.$$ This implies that $g'(0)=0,$ where
$$g'(0)=\frac{P_1(x_0)h(x_0)\int\limits_{-1}^{1}P^*(x)dx-P^*(x_0)\int\limits_{-1}^{1}P_1(x)h(x)dx}{\left(\int\limits_{-1}^{1}P^*(x)dx\right)^2},$$
and so $$\int\limits_{-1}^{1}P_1(x)(P_2
(x)h(x_0)-P_2(x_0)h(x))dx=0,$$ for all polynomials $h$ with $\deg
{(h)}\le \deg{(P_2)}.$ Observe, that this equality implies that if $l(x)$ is such that

$$l(x)(x-x_0)=P_2 (x)h(x_0)-P_2(x_0)h(x),$$ then if $h(x)$ runs over
all polynomials of degree $\le \deg{(P_2)},$ then $l(x)$ runs over
all polynomials with  $\deg{(l)}\le \deg{(P_2)} -1.$ Therefore,
\begin{equation}\label{linear} \int\limits_{-1}^{1}P_1(x)(x-x_0)l(x)dx=0
\end{equation} holds for all polynomials $l(x)$ of degree $\le \deg{(P_2)} -1.$

 If $\deg{(P_2)}\ge 2$ take $l(x)=x-x_0$ to get a contradiction
 (integral of a nonnegative non-zero function cannot be equal to $0$).
 Now, suppose that $\deg {(P_2)}=1.$ Then
 $$\int\limits_{-1}^{1}P_1(x)(x-x_0)dx=0$$ and one can write
 $P^*(x)=(a-x)P_1 (x)$ where $a>1$ or $P^*(x)=(b+x)P_1 (x)$ for some
 $b>1$. In both of these cases it is easy to see that $S(P^*,x_0)=S
 (P_1,x_0)$. Indeed, in the first case
\begin{align*}
 S(P^*,x_0)&= \frac{(a-x_0)P_1(x_0)}{\int\limits_{-1}^{1}(a-x)P_1(x)dx}\\&=\frac{(a-x_0)P_1(x_0)}{\int\limits_{-1}^{1}(x_0-x)P_1(x)dx+(a-x_0)\int_{-1}^1P_1(x)dx}=S(P_1,x_0).
  \end{align*}
 In the second case, it can be done in the same way. But then, taking $$P_3 (x)=(1+x)P_1(x)$$ and $$P_4
 (x)=(1-x)P_1(x)$$ and repeating all arguments from the
 beginning of the proof one get that either $S(P_3,x_0)$ or $S(P_4,x_0)$ is
 not less then $S(P,x_0)=S(P_1,x_0)$ and all zeros of $P_3$ and $P_4$ lie in
 the segment $[-1,1],$ that contradicts our assumption. Hence, all zeros of $P^*(x)$ lie in the interval $[-1,1].$

 We distinguish two cases depending on parity of $n.$

 If $n=2k+1,$ $k\ge 0$ an extremal polynomial can be expressed in one of the following ways: $P^*(x)=(1+x)g^2
 (x)$ or $P^*(x)=(1-x)g^2_1(x)$. If $n=2k,$ then an extremal polynomial can be expressed as $P^*(x)=(1-x^2)g^2(x)$ or $P^*(x)=g^2(x).$ In general, we can write an extremal polynomial as $P^*(x)=w(x)g^2(x),$ where $w(x)$ is one of the function $1-x, 1+x, 1-x^2, 1.$

  For any fixed polynomial
$h(x)$ with $\deg {(h)}\le \deg {(g)}$ consider the function
$$\psi (t)=\frac {w(x_0)(g(x_0)+th(x_0))^2}{\int\limits_{-1}^{1} w(x)(g(x)+th(x))^2dx}.$$
Since $P^*$ is extremal, this function has a local maximum at $t=0,$ and so $\psi'
(0)=0,$ i.e.,
\begin{equation}
\label{ex:klurman}
\psi'(0)=2w(x_0)\cdot\frac{g(x_0)h(x_0)\int\limits_{-1}^{1}w(x) g^2 (x)dx- g^2(x_0)\int\limits_{-1}^{1}w(x)g(x)h(x)dx}{\left(\int\limits_{-1}^{1}w(x)g^2(x)dx\right)^2}=0.
\end{equation}
Since $g(x_0)\ne 0$ (otherwise, $\psi(0)=0,$ which contradicts to maximality of $P^*$) last equality implies
$$h(x_0)\int\limits_{-1}^{1}w(x) g^2 (x)dx-
g(x_0)\int\limits_{-1}^{1}w(x)g(x)h(x)dx=0$$
or
\begin{equation}
\label{der:klurman}
\int\limits_{-1}^{1}w(x)g(x)(h(x_0)g(x)-h(x)g(x_0))dx=0
 \end{equation}for all polynomials  $h(x)\in\mathbb{P}_k$ if $w(x)=1,1-x,1+x$ and for all polynomials $h(x)\in\mathbb{P}_{k-1},$ if $w(x)=1-x^2$.
 We first consider the case $w(x)=1,1-x,1+x.$
 Repeating the same
 argument as we used to prove~\eqref{linear} we can deduce that~\eqref{der:klurman} implies that for all $l\in\mathbb{P}_{k-1}$ we have $$\int\limits_{-1}^{1}w(x)g(x)(x-x_0)l(x)dx=0.$$
 Denote $$G(x)=(x-x_0)g(x)$$ and consider the sequence of
 polynomials $p_k$ orthonormal on $[-1,1]$ with
 respect to the weight $w(x).$ Since $\deg {(G)}=k+1$ and orthonormal polynomials of degree $\le k+1$ form a basis (over $\mathbb{R}$) of $\mathbb{P}_{k+1},$ one can write $$G(x)=\sum\limits_{m=0} ^{k+1} c_m p_m
 (x)$$ for some real constants $c_m.$ Taking
 $l(x)=p_i(x)$ for $0\le i\le k-1$ we obtain that $c_i=0$ for $0\le i \le
 k-1$. Indeed, if $l(x)=p_i(x),$ $0\le i\le k-1$ then $$\int\limits_{-1}^{1}w(x)G(x)J_i(x)dx=0=\sum_{m=0}^{k+1}c_m\int\limits_{-1}^{1}w(x)p_k(x)p_i(x)dx=c_i.$$ Thus, $$G(x)=(x-x_0)g(x)=c_{k+1} p_{k+1} (x) +c_k p_k (x).$$
 Letting $x=x_0,$ we get $c_{k+1} p_{k+1} (x_0) +c_k p_k (x_0)=0,$ and so $$g(x)=g_{extr} (x):=c\frac{p_{k+1} (x)p_k (x_0)-p_{k+1} (x_0)p_k
 (x)}{x-x_0},$$
 for some real constant $c.$
  In case $w(x)=1-x^2,$ we have to take $k-1$ instead $k.$ It gives us a polynomial $g,$ of the form
  $$g(x)=g_{extr} (x):=c\frac{p_{k} (x)p_{k-1} (x_0)-p_{k} (x_0)p_{k-1}
 (x)}{x-x_0},\mbox\ \    k\ge 1.$$


Now, using Christoffel-Darboux's  formula (see \cite{MR0372517})
 $S(P,x_0)$ can be computed explicitly. Indeed,
 \begin{align*}
 \int\limits_{-1}^1w(x)\left(\frac{p_{k+1}(x)p_k(x_0)-p_{k+1}(x_0)p_k(x)}{x-x_0}\right)^2dx &=\frac{\gamma_{k}^2}{\gamma_{k+1}^2}
 \int_{-1}^1w(x)\sum_{i=0}^k(p_i(x_0)p_i(x))^2dx \\&=\frac{\gamma_{k}^2}{\gamma_{k+1}^2}\sum_{i=0}^k p_i(x_0)^2, \end{align*}
 hence

\begin{align*}
S(P,x_0)&=w(x_0)\frac{(g_{extr}(x_0))^2}{\int\limits_{-1}^1w(x)(g_{extr}(x))^2dx}\\&=w(x_0)\frac{(\frac{\gamma_{k+1}}{\gamma_k})^2(\sum\limits_{l=0} ^{k} p_l ^2 (x_0))^2}
{\int\limits_{-1}^1w(x)\left(\frac{p_{k+1}(x)p_k(x_0)-p_{k+1}(x_0)p_k(x)}{x-x_0}\right)^2dx}
\\&=w(x_0)\left(\sum\limits_{l=0}^{k}
p_l^2 (x_0)\right). \end{align*} In case, when $n=2k$ and $w(x)=1-x^2$ we get $$S(P,x_0)=w(x_0)\left(\sum\limits_{l=0}^{k-1}
p_l^2 (x_0)\right).$$

Let $n=2k+1.$ If $w(x)=1+x,$ then $$p_k(x)=J_k^{(1,0)}(x),$$ where $J_k^{(1,0)}(x)$ is the Jacobi polynomial associated with weight $(1+x)(1-x).$ Hence,  $$S(P,x_0)=(1+x_0)\sum\limits_{l=0}^{k} (J^{(0,1)}_l (x_0))^2=S_k(x_0).$$ By analogy, if $w(x)=1-x,$ then

$$S(P,x_0)=(1-x_0)\sum\limits_{l=0}^{k} (J^{(1,0)}_l (x_0))^2=S_k(-x_0).$$

In this way, we get sharp pointwise inequality
\begin{equation}\label{klurbern2}P_{2k+1}(x_0)\le \max{(S_k(x_0),S_k (-x_0))}
\int\limits_{-1}^{1} P_{2k+1}(x)dx \end{equation} for all
$P_{2k+1}\in\mathbb{P}^+_{2k+1}.$

In case $n=2k,$ $w(x)=1$ or $w(x)=1-x^2$ and $$S(P,x_0)=F_k(x_0)$$ or $$S(P,x_0)=H_k(x_0)$$ respectively, where $H_k$ and $F_K.$
We arrive at the following sharp pointwise inequality:
 \begin{equation}\label{klurbern1}
 P_{2k}(x_0)\le \max{(F_k(x_0),H_k (x_0))}
\int\limits_{-1}^{1} P_{2k}(x)dx \end{equation} for all
$P_{2k}\in\mathbb{P}^+_{2k}.$

Let $P_n$ be a polynomial of degree $n,$ that is monotone on $[-1,1],$ i.e., $P_n\in\Delta_n$
Then, $P_n'$ is a nonnegative polynomial on $[-1,1].$ Note,
 that \[\int\limits_{-1}^{1} P_n'(x)dx=P_n(1)-P_n(-1)\le 2\|P_n\|. \] Combining last inequality with~\eqref{klurbern2} and~\eqref{klurbern1} we get
\begin{equation}\label{derivpoint1}
P'_{2k+2}(x)\le 2\cdot \max{(S_k(x),S_k (-x))}||P_{2k+2}|| ,\end{equation}
and \begin{equation}\label{derivpoint2}
P'_{2k+1}(x)\le 2\cdot \max{(H_k(x),F_k
(x))}||P_{2k+1}||,\end{equation} for all monotone polynomials $P_n.$
 This completes the proof.
\end{proof}


Using  Theorem~\ref{oklurman} one can give an alternative proof of
Bernstein's result, that is Theorem~\ref{bernquaz} for polynomials of even degree. The
following fact about orthogonal polynomials is needed.

\begin{theorem}( Szeg{\H{o}} \cite{MR0372517}, 1919)\label{szege}. Let $w(x)$ be a weight function which is non- decreasing (non-
increasing) in the interval $[a,b]$, $b$  and $a$ are finite. If
$\{p_n\}$ is the set of the corresponding orthogonal polynomials, the
functions $w(x)p_n (x)^2 $ attain their maxima in $[a,b]$ at
$x=b$ ($x=a$).
\end{theorem}

{\it Proof of~Theorem~\ref{bernquaz}.}
We consider an even case $n=2k+2,$ $k\ge 0.$ Using Szeg{\H{o}}'s theorem for
non-decreasing weight $w(x)=1+x$ and for non-increasing weight
$w(x)=1-x$ together with the fact that $$J^{(0,1)}_l(1)=J^{(1,0)}_l(-1)=\frac{\sqrt{l+1}}{\sqrt{2}}$$ we get:
$$(1+x)(J^{(0,1)}_l (x))^2\le 2(J^{(0,1)}_l (1))^2=l+1,$$
$$(1+x)(J^{(1,0)}_l (x))^2\le 2(J^{(1,0)}_l (-1))^2=l+1$$
for all $l\ge 0$ and $x\in[-1,1]$. Summing these inequalities for $0\le l\le k$ and
using Theorem~\ref{oklurman} we get

$$P'_n(x)\le 2\sum\limits_{l=0}^k(1+l) \|P_n\|=\frac{n(n+2)}{4}\|P_n\|,$$
 that proves Bernstein-Markov's inequality for monotone polynomials
of even degree.

Multiplying both sides of \eqref{derivpoint1} and \eqref{derivpoint2} by $\sqrt{1-x^2}$ and taking supremum over all $x\in[-1,1]$ we get the following
\begin{corollary}{ (Sharp Benstein-type inequality for monotone polynomials).}
\begin{equation}\label{klurbern}
 \sup_{P_n\in\triangle_
n}\frac{\|P'_n (x)\sqrt{1-x^2}\|}{\|P_n\|}=\left\{
\begin{array}{ll}
 2\|\sqrt{1-x^2}S_k(x)\| , & \mbox{\rm } n=2k+2 ,\\
 2\max({\|\sqrt{1-x^2}H_k(x)\|,\|\sqrt{1-x^2}F_k(x)\|}), & \mbox{\rm  }  n=2k+1.
\end{array}
\right.
\end{equation}
\end{corollary}

Using estimates for Jacobi polynomials one can observe that the right hand side is of the order $\frac{2}{\pi}n.$ This implies that Bernstein's inequality for monotone polynomials is not essentially better than the 
classical one. 
 \begin{remark}
From the proof of~Theorem\eqref{oklurman} it follows that equality in~\eqref{klurbern} holds for one of the following polynomials
\begin{align*}
&s_k (x):=\int_{-1}^x(1+t)\sum\limits_{l=0}^{k} (J^{(0,1)}_l (t))^2;\\&
h_k (x):=\int_{-1}^x(1-t^2)\sum\limits_{l=0}^{k-1} (J_l ^{(1,1)} (t))^2;\\&
f_k(x):=\int_{-1}^x\sum\limits_{l=0}^{k} (J_l ^{(0,0)} (t)^2,\\
\end{align*}
that are normalized, such that
\begin{align*}
&s_k(-1)=-s_k(1),\\& h_k(-1)=-h_k(1),\\& f_k(-1)=-f_k(1).\\
\end{align*}
\end{remark}
\end{section}
\bibliographystyle{plain}

\bibliography{mybliography}

\end{document}